\documentclass[12pt,reqno]{amsart}
\usepackage{amsmath,amssymb,amsthm,graphics}
\usepackage[all]{xy}   
\usepackage{hyperref}

\newtheorem{theorem}{Theorem}

\newtheorem*{MatuiTheorem}{Matui's Theorem}
\newtheorem{proposition}[theorem]{Proposition}
\newtheorem{lemma}[theorem]{Lemma}

\theoremstyle{definition}
\newtheorem{definition}[theorem]{Definition}
\newtheorem*{definition*}{Definition}

\newtheorem{note}[theorem]{Note}

\newtheorem{example}[theorem]{Example}

\newcommand{\newword}[1]{\textbf{#1}}

\newcommand{\R}{\mathcal{R}}

\newcommand{\tast}{\text{\small$*$}}
\newcommand{\emptystring}{\varepsilon}

\usepackage{color}

\definecolor{orange}{rgb}{1,0.5,0}

\begin{document}

\title{On the Asynchronous Rational Group}

\author{James Belk}
\address{Mathematics Program \\ Bard College \\ Annandale-on-Hudson, New York 12504 \\ USA}
\email{belk@bard.edu}
\urladdr{http://faculty.bard.edu/belk}

\author{James Hyde}
\address{School of Mathematics and Statistics \\ University of St Andrews \\ St Andrews, UK}
\email{jameshydemaths@gmail.com}
\urladdr{http://www-groups.mcs.st-andrews.ac.uk/~jhyde/}
\thanks{The second author's work was partially supported by an Engineering and Physical Sciences Research Council (EPSRC) PhD studentship}

\author{Francesco Matucci}
\address{Instituto de Matem\'{a}tica, Estat\'{i}stica e Computa\c c\~{a}o Cient\'{i}fica \\ Universidade Estadual de Campinas (UNICAMP) \\ S\~{a}o Paulo, Brasil}
\email{francesco@ime.unicamp.br}
\urladdr{https://www.sites.google.com/site/francescomatucci}
\thanks{The third author is a member of the Gruppo Nazionale per le Strutture Algebriche, Geometriche e le loro Applicazioni (GNSAGA) of the Istituto Nazionale di Alta Matematica (INdAM) and gratefully acknowledges the support of the 
Funda\c{c}\~ao de Amparo \`a Pesquisa do Estado de S\~ao Paulo 
(FAPESP Jovens Pesquisadores em Centros Emergentes grant 2016/12196-5),
of the Conselho Nacional de Desenvolvimento Cient\'ifico e Tecnol\'ogico (CNPq 
Bolsa de Produtividade em Pesquisa PQ-2 grant 306614/2016-2) and of
the Funda\c{c}\~ao para a Ci\^encia e a Tecnologia  (CEMAT-Ci\^encias FCT project UID/Multi/04621/2013).}

\subjclass[2010]{Primary 20F65; Secondary 20E32, 20F05, 20F10, 68Q70}
\begin{abstract}
We prove that the asynchronous rational group defined by Grigorchuk, Nekrashevych, and Sushchanskii is simple and not finitely generated.  Our proofs also apply to certain subgroups of the rational group, such as the group of all rational bilipschitz homeomorphisms.
\end{abstract}

\maketitle

In \cite{GNS}, Grigorchuk, Nekrashevych, and Sushchanskii defined the group $\R$ of all asynchronous rational homeomorphisms of the Cantor set (denoted~$\mathcal{Q}$ in~\cite{GNS}).  The purpose of this note is to prove the following theorem.

\begin{theorem}\label{thm:MainTheorem}$\R$ is simple but not finitely generated.
\end{theorem}

Our proof of simplicity follows the general outline of Epstein~\cite{Epstein}, and resembles similar proofs for the simplicity of Thompson's group~$V$, the higher-dimensional Thompson groups~$nV$ defined by Brin~\cite{Brin}, and R\"{o}ver's group $V\mathcal{G}$~\cite{Rover}, though it uses a new trick involving the construction of the element
\[
(f,(f,(f,(f,\ldots))))
\]
from a rational homeomorphism~$f\in\R$.  A similar idea was used by Anderson~\cite{And} to prove that the full group of homeomorphisms of the Cantor set is simple.

The proof that $\R$ is not finitely generated involves the primes that divide the lengths of cycles in the transducer for an element~$f\in\R$.  Specifically, we generate a sequence $\{f_p\}$ of elements of $\R$ (with $p$ prime) that no finitely generated subgroup of~$\R$ can contain.

The techniques of both of our proofs apply not just to $\mathcal{R}$ but also to certain interesting subgroups of~$\mathcal{R}$.  For example:

\begin{theorem}The group of all rational bilipschitz homeomorphisms of\/ $\{0,1\}^\omega$ is simple but not finitely generated.
\end{theorem}

\section{Background and Notation}

Let $\{0,1\}^\omega$ be the Cantor set of all infinite binary sequences, and let $\{0,1\}^\tast$ be the set of all finite binary sequences, including the empty sequence~$\emptystring$.

\begin{definition}
An \textbf{asynchronous binary transducer} 
is a quadruple $(S, s_0, t,o)$, where
\begin{enumerate}
\item $S$ is a finite set (the set of \textbf{internal states} of the transducer);\smallskip
\item $s_0$ is a fixed element of $S$ (the \textbf{initial state});\smallskip
\item $t$ is a function $S \times \{0,1\} \to S$ (the \textbf{transition function}); and\smallskip
\item $o$ is a function $S \times \{0,1\} \to \{0,1\}^\tast$ (the \textbf{output function}).
\end{enumerate}
\end{definition}

\begin{example}Figure~\ref{fig:TransducerExample} shows the \newword{state diagram} for a certain asynchronous transducer.  This is a directed graph with one node for each state~$s\in S$, and a directed edge from $s$ to $s'$ if $t(s,\sigma) = s'$ for some $\sigma\in\{0,1\}$.  We label such an edge by the pair $\sigma\,|\, o(s,\sigma)$, with two labels in the case where $t(s,0) = t(s,1) = s'$.  For example, the transducer in Figure~\ref{fig:TransducerExample} has two states $s_0,s_1$, with transition function
\[
t(s_0,0)=s_1,\quad t(s_0,1)=s_0,\quad t(s_1,0)=s_0,\quad t(s_1,1)=s_0
\]
and output function
\[
o(s_0,0)=\emptystring,\quad o(s_0,1)=11,\quad o(s_1,0)=0,\quad o(s_1,1)=10\tag*{\qedsymbol}
\]
\end{example}
\begin{figure}[b]
\centering
\includegraphics{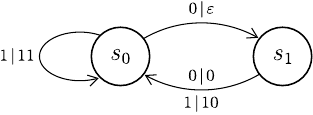}
\caption{The state diagram for an asynchronous transducer.}
\label{fig:TransducerExample}
\end{figure}

\pagebreak
If $T = (S,s_0,t,o)$ is a transducer, $s\in S$ is a state, and $\sigma_1\sigma_2\cdots$ is a binary sequence (finite or infinite), the corresponding \newword{sequence of states} $\{s_n\}$ is defined recursively by $s_1 = s$ and
\[
s_{n+1} = t(s_n,\sigma_n)
\]
for all $n\geq 1$. We define
\[
t(s,\sigma_1\cdots \sigma_n) = s_{n+1}\quad\;\;\text{and}\quad\;\; o(s,\sigma_1\cdots\sigma_n) = o(s_1,\sigma_1)\cdots o(s_n,\sigma_n)
\]
for each $n$, where the product on the right is a concatenation of binary sequences.  This extends the functions $t$ and $o$ to functions
\[
t\colon S\times \{0,1\}^\tast \to S\qquad\text{and}\qquad o\colon S\times\{0,1\}^\tast \to \{0,1\}^\tast.
\]
If $\sigma_1\sigma_2\cdots$ is an infinite binary sequence, we can also define
\[
o(s,\sigma_1\sigma_2\cdots) = o(s_1,\sigma_1)\,o(s_2,\sigma_2)\,\cdots
\]
where the product on the right is an infinite concatenation of binary sequences.

\begin{definition}A homeomorphism $f\colon \{0,1\}^\omega\to\{0,1\}^\omega$ is \newword{rational} if there exists an asynchronous binary transducer $(S,s_0,t,o)$ such that $f(\psi) = o(s_0,\psi)$ for all $\psi\in\{0,1\}^\omega$.
\end{definition}

In \cite{GNS}, Grigorchuk, Nekrashevych and Sushchanskii proved that the set $\mathcal{R}$ of all rational homeomorphisms of $\{0,1\}^\omega$ forms a group under composition.  This is the \newword{(asynchronous) rational group}~$\mathcal{R}$.

\begin{note}Though we are restricting ourselves to a binary alphabet, there is a rational group $\mathcal{R}_A$ associated to any finite alphabet $A$ with at least two symbols.  It was proven in~\cite{GNS} that the isomorphism type of $\mathcal{R}_A$ does not depend on the size of~$A$, so it suffices to consider only the binary case.
\end{note}

Grigorchuk, Nekrashevych, and Sushchanskii also gave a useful test for determining whether a homeomorphism is rational.  If $\alpha\in\{0,1\}^\tast$ is any finite binary sequence, let $I_\alpha$ denote the subset of $\{0,1\}^\omega$ consisting of all sequences that start with~$\alpha$. Note then that $\{I_\alpha \mid \alpha\in\{0,1\}^\tast\}$ is a basis of clopen sets for the topology on $\{0,1\}^\omega$.

Given a homeomorphism $f\colon \{0,1\}^\omega \to \{0,1\}^\omega$ and an $\alpha\in\{0,1\}^\tast$, the \newword{restriction} of $f$ to~$\alpha$ is the map $f|_\alpha\colon \{0,1\}^\omega\to\{0,1\}^\omega$ defined by
\[
f(\alpha\omega) = \beta f|_\alpha(\omega)
\]
where $\beta$ is the greatest common prefix of all sequences in~$f(I_\alpha)$.

\begin{theorem}\label{thm:FinitelyManyRestrictions} Let $f\colon\{0,1\}^\omega\to\{0,1\}^\omega$ be a homeomorphism.  Then $f$ is rational if and only if $f$ has only finitely many different restrictions.
\end{theorem}
\begin{proof}This follows from \cite[Theorem~2.5]{GNS}, since a homeomorphism cannot have any empty restrictions.
\end{proof}

\section{Simplicity}

In \cite{Epstein}, Epstein introduced a general framework for proving that a group $G$ of homeomorphisms is simple.  The first step is to use some variant of Epstein's commutator trick to prove that the commutator subgroup $[G,G]$ is simple, and then one must give an independent proof that $G = [G,G]$.

We start by observing a few important properties of $\R$.

\begin{definition}Let $G$ be a group of homeomorphisms of a topological space~$X$.
\begin{enumerate}
\item We say that a homeomorphism $h$ of $X$ \newword{locally agrees with~$\boldsymbol{G}$} if for every point $p\in X$, there exists a neighborhood $U$ of $p$ and a $g\in G$ such that $h|_U = g|_U$.\smallskip
\item We say that $G$ is \newword{full} if every homeomorphism of $X$ that locally agrees with $G$ belongs to~$G$.
\end{enumerate}
\end{definition}

That is, $G$ is full if one can determine whether a homeomorphism $h$ lies in~$G$ by inspecting the germs of~$h$.  The word ``full'' here comes from the theory of \'{e}tale groupoids, where $G$~is full if and only if $G$ is the ``full group'' of the \'{e}tale groupoid consisting of all germs of elements of~$G$.

Examples of full groups include the full homeomorphism group of any topological space, the full group of diffeomorphisms of any differentiable manifold, and the Thompson groups $F$, $T$, and $V$ acting on the interval, the circle, and the Cantor set, respectively.  Other Thompson-like groups such as R\"{o}ver's group $V\mathcal{G}$ (see~\cite{Rover}) and Brin's higher-dimensional Thompson groups~$nV$ (see~\cite{Brin}) are also full. 

\begin{proposition}\label{prop:LocallyBelongsCantorSet}Let $G$ be a group of homeomorphisms of the Cantor set~$\{0,1\}^\omega$, and let $h$ be a homeomorphism of~$\{0,1\}^\omega$.  Then $h$ locally belongs to $G$ if and only if there exists a partition
\[
\{0,1\}^\omega = I_{\alpha_1} \uplus \cdots \uplus I_{\alpha_n}
\]
and elements $g_1,\ldots,g_n\in G$ such that $h$ agrees with $g_i$ on each $I_{\alpha_i}$.
\end{proposition}
\begin{proof}Clearly $h$ locally belongs to $g$ if it satisfies the given condition.  For the converse, suppose that $h$ locally belongs to~$g$.  Since the $I_\alpha$ form a basis for the topology on $\{0,1\}^\omega$ and $\{0,1\}^\omega$ is compact, there exists a finite cover $\{I_{\alpha_1},\ldots, I_{\alpha_n}\}$ of $\{0,1\}^\omega$ and elements $g_1,\ldots,g_n\in G$ so that $h$ agrees with $g_i$ on each~$I_{\alpha_i}$.  Since any two $I_{\alpha_i}$ are either disjoint or one is contained in the other, we may assume that the cover $\{I_{\alpha_1},\ldots, I_{\alpha_n}\}$ is a partition of~$\{0,1\}^\omega$.
\end{proof}

\begin{proposition}\label{prop:RationalGroupFull}The rational group $\R$ is full.
\end{proposition}
\begin{proof}Let $h$ be a homeomorphism of $\{0,1\}^\omega$, and suppose that $h$ locally belongs to~$\R$.  By Proposition~\ref{prop:LocallyBelongsCantorSet}, there exists a partition
\[
\{0,1\}^\omega = I_{\alpha_1}\uplus \cdots\uplus I_{\alpha_n}
\]
and elements $g_1,\ldots,g_n\in\R$ such that $h$ agrees with $g_i$ on each~$\alpha_i$.  Then $h|_\alpha = g_i|_\alpha$ whenever $\alpha_i$ is a prefix of~$\alpha$, so all but finitely many restrictions of $h$ are also restrictions of some~$g_i$.  Since each $g_i$ is rational, each $g_i$~has finitely many different restrictions by Theorem~\ref{thm:FinitelyManyRestrictions}, and therefore $h$ has finitely many different restrictions as well.  Then $h\in\R$ by Theorem~\ref{thm:FinitelyManyRestrictions}.
\end{proof}

We also require a certain transitivity property.

\begin{definition}Let $G$ be a group of homeomorphisms of the Cantor set.  We say that $G$ is \newword{flexible} if for every pair $E_1,E_2$ of proper, nonempty clopen subsets of~$X$, there exists a $g\in G$ so that $g(E_1)\subseteq E_2$.
\end{definition}

Examples of flexible groups include the full homeomorphism group of the Cantor set, Thompson's groups $T$ and~$V$, and many other Thompson-like groups, such as R\"{o}ver's group~$V\mathcal{G}$ and Brin's groups~$nV$. 

\begin{proposition}\label{prop:RationalGroupFlexible}The rational group $\R$ is flexible.
\end{proposition}
\begin{proof}As observed in~\cite{GNS}, $\R$~contains Thompson's group~$V$, and since $V$ is flexible $\R$ must be flexible as well.
\end{proof}

Next we need some notation.  Given two homeomorphisms $f,g\in\mathcal{R}$, let $(f,g)$ denote the homeomorphism
\[
(f,g)(\omega) = \begin{cases} 0\,f(\zeta) & \text{if }\omega = 0\zeta, \\ 1\,f(\zeta) & \text{if }\omega=1\,\zeta.\end{cases}
\]
Note that $(f,g)|_{0\alpha} = f|_\alpha$ and $(f,g)|_{1\alpha} = g|_\alpha$ for all $\alpha\in\{0,1\}^\tast$, so by Theorem~\ref{thm:FinitelyManyRestrictions} $(f,g)$ is rational.

Note also that $f = (f|_0,f|_1)$ for any $f\in\R$ such that $f(I_0)=I_0$ and $f(I_1)=I_1$. In particular, any element of $\R$ that is the identity on $I_1$ can be written as $(g,1)$ for some $g\in\R$, and similarly any element of~$\R$ that is the identity on $I_0$ can be written as $(1,g)$ for some $g\in\R$.

The following lemma lets us define the element
\[
(f,(f,(f,(f,\ldots))))
\]
for a given~$f\in\R$.\pagebreak

\begin{lemma}\label{lem:HilbertHotel}Let $f\in\R$.  Then there exists a $g\in\R$ such that $g=(f,g)$.
\end{lemma}
\begin{proof}Let $(S,s_0,t,o)$ be a transducer for~$f$, and consider the transducer $(S',s_0',t',o')$ defined as follows:
\begin{enumerate}
\item $S'$ is obtained from $S$ by adding a new state $s'_0$.\smallskip
\item $t'$ agrees with $t$ on $S\times\{0,1\}$, and satisfies $t(s'_0,0) = s_0$ and $t(s'_0,1) = s'_0$.\smallskip
\item $o'$ agrees with $o$ on $S\times\{0,1\}$, and satisfies $o(s'_0,0) = 0$ and $o(s'_0,1)=1$.
\end{enumerate}
It is easy to check that $(S',s_0',t',o')$ defines a rational homeomorphism $g$ of $\{0,1\}^\omega$, and that $g = (f,g)$.
\end{proof}

We say that an element $f\in\R$ has \newword{small support} if there exists a proper, nonempty clopen subset $E$ of $\{0,1\}^\omega$ such that $f$ is the identity on the complement of~$E$.  In this case, we say that $f$ is \newword{supported} on~$E$.

\begin{lemma}\label{lem:CommutatorLemma}Every element of $\R$ with small support lies in $[\R,\R]$.
\end{lemma}
\begin{proof} Let $f\in\R$ have small support.  Since $\R$ is flexible by Proposition~\ref{prop:RationalGroupFlexible}, we can conjugate $f$ by an element of $\mathcal{R}$ so that it is supported on~$I_{01}$. Then
\[
f = \bigl((1,g),1\bigr)
\]
for some $g\in\R$. By Lemma~\ref{lem:HilbertHotel}, there exists an $h\in\R$ that satisfies the equation
\[
h = (g,h).
\]
Let $k = (1,h)$, and let $x_0\in\R$ be the first generator for Thompson's group~$F$, i.e.~the homeomorphism $x_0\colon\{0,1\}^\omega\to\{0,1\}^\omega$ satisfying
\[
x_0(00\zeta) = 0\zeta,\qquad x_0(01\zeta) = 10\zeta,\qquad x_0(1\zeta)=11\zeta
\]
for all $\zeta\in\{0,1\}^\omega$.
(A trandsucer for $x_0$ is shown in Figure~\ref{fig:Transducerx0}.) Then
\begin{figure}[b]
\centering
\qquad\includegraphics{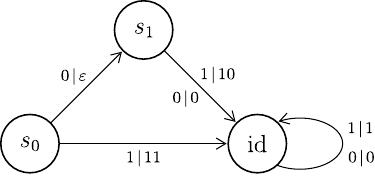}
\caption{The transducer for the element~$x_0$.}
\label{fig:Transducerx0}
\end{figure}
\[
x_0^{-1}k x_0 = x_0^{-1}(1,h)x_0 = x_0^{-1}\bigl(1,(g,h)\bigr)x_0 \,=\, \bigl((1,g),h\bigr) \,=\, fk
\]
and therefore $f = x_0^{-1}k x_0k^{-1}$.
\end{proof}

\begin{proposition}\label{prop:SmallSupportGenerate}The elements of small support generate $\R$, and therefore~$\R=[\R,\R]$.
\end{proposition}
\begin{proof}Let $f$ be any non-identity element in $\R$.  Then there exists a nonempty clopen set $E\subseteq\{0,1\}^\omega$ such that $f(E)$ is disjoint from~$E$ and $E\cup f(E)$ is not the whole Cantor set. Let $g$ be the homeomorphism of $\{0,1\}^\omega$ that agrees with $f$ on~$E$, agrees with $f^{-1}$ on~$f(E)$, and is the identity elsewhere.  Since $\mathcal{R}$ is full by Proposition~\ref{prop:RationalGroupFull}, we know that $g\in\mathcal{R}$.  Then $g$ is supported on $E\cup f(E)$, and $gf$ is the identity on~$E$, so $g$ and $gf$ have small support.  Then $f=g^{-1}(gf)$ is a product of elements of small support, and is therefore in $[\R,\R]$ by Lemma~\ref{lem:CommutatorLemma}.
\end{proof}

\begin{theorem}\label{thm:IsSimple}$\R$ is simple.
\end{theorem}
\begin{proof}Let $N$ be a nontrivial normal subgroup of $\R$, and let $f_0$ be a nontrivial element of~$N$.  Then there exists a nonempty clopen set $E\subseteq\{0,1\}^\omega$ such that $f_0(E)$ is disjoint from~$E$.  Since $\mathcal{R}$ is flexible by Proposition~\ref{prop:RationalGroupFlexible}, there exists a $c\in\mathcal{R}$ so that $c(I_0)\subseteq E$. Then $f_1 = c^{-1}f_0c$ is an element of~$N$ and has the property that $f_1(I_0)$ is disjoint from~$I_0$.

Let $g,h\in\R$.  Then $f_1(g,1)^{-1}f_1^{-1}$ is supported on~$f(I_0)$, so the element
\[
g' = (g,1)f_1(g,1)^{-1}f_1^{-1}
\]
of~$N$ agrees with $(g,1)$ on~$I_0$.  It follows that $\bigl[g',(h,1)\bigl] = \bigl([g,h],1\bigr)$, so the latter is in $N$.

Since $[\R,\R] = \R$ by Proposition~\ref{prop:SmallSupportGenerate}, it follows that $(k,1)\in N$ for all $k\in\R$.   Since $\mathcal{R}$ is flexible, we can conjugate by elements of~$\mathcal{R}$ to deduce that every element of $\R$ with small support lies in~$N$. But such elements generate~$\R$, and therefore~$N=\R$.
\end{proof}

Note that the argument in Lemma~\ref{lem:CommutatorLemma}, Proposition~\ref{prop:SmallSupportGenerate}, and Theorem~\ref{thm:IsSimple} applies to many other groups as well. Indeed, we have proven the following.

\begin{theorem}Let $G$ be a full, flexible group of homeomorphisms of the Cantor set.  Suppose that:
\begin{enumerate}
\item For all $g\in G$ both $(1,g)$ and $(g,1)$ lie in $G$, and every element of $G$ supported on $I_0$ or $I_1$ has this form,\smallskip
\item For all $g\in G$ there exists an $h\in G$ so that $h = (g,h)$, and\smallskip
\item The first generator $x_0$ of Thompson's group $F$ lies in~$G$.
\end{enumerate}
Then $G$ is simple.\hfill\qedsymbol
\end{theorem}

\noindent Groups to which this theorem applies include:
\begin{itemize}
\item The group of all homeomorphisms of the Cantor set~(see~\cite{And}).\smallskip
\item The group of all computable homeomorphisms of the Cantor set.\smallskip
\item The group of all bilipschitz homeomorphisms (or computable bilipschitz homeomorphisms) of the Cantor set.\smallskip
\item The group of all bilipschitz elements of~$\mathcal{R}$.
\end{itemize}
Note that the group of bilipschitz elements of~$\R$ is a proper subgroup of~$\R$, since for example it does not contain the homeomorphism whose transducer is shown in Figure~\ref{fig:TransducerExample}.

Incidentally, an alternative to the proof of Theorem~\ref{thm:IsSimple} above is the following remarkable theorem:

\begin{MatuiTheorem} Let $G$ be a full, flexible group of homeomorphisms of the Cantor set.  Then\/ $[G,G]$ is simple.
\end{MatuiTheorem}
\begin{proof}See \cite[Theorem~4.16]{Matui}. The word ``flexible'' does not appear in Matui's work, but Matui does prove in \cite[Proposition~4.11]{Matui} that the full group of an essentially principal \'{e}tale groupoid is purely infinite and minimal if and only if the associated full group is flexible as defined above. Technically, Matui assumes that the \'{e}tale groupoid of germs is Hausdorff, but nothing in his proof requires this condition. \end{proof}

In addition to the groups considered above, this theorem applies to many groups of interest, including the full homeomorphism group of the Cantor set, Thompson's groups~$V$ as well as the generalized groups $V_{n,r}$ (not all of which are simple), R\"{o}ver's group $V\mathcal{G}$ (see~\cite{Rover}), and Brin's higher-dimensional Thompson groups~$nV$ (see~\cite{Brin}).

\section{Finite Generation}

For each prime $p$, let $f_p$ be the element of $\mathcal{R}$ defined by the transducer shown in Figure~\ref{fig:TransducerP}.  This homeomorphism switches every $p$'th digit of a binary sequence, leaving the remaining digits unchanged.  The goal of this section is to prove the following theorem.
\begin{figure}[t]
\centering
\includegraphics{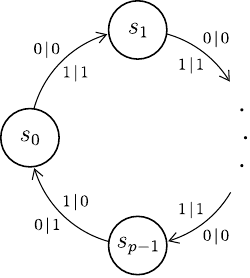}
\caption{The transducer for~$f_p$.}
\label{fig:TransducerP}
\end{figure}

\begin{theorem}\label{thm:NotFG}Let $\mathcal{M}$ be any submonoid of $\mathcal{R}$ that contains $f_p$ for infinitely many primes~$p$.  Then $\mathcal{M}$ is not finitely generated.
\end{theorem}

Since $\mathcal{R}$ is itself such a submonoid, it follows from this theorem that $\mathcal{R}$ is not finitely generated.  (Note that a group is finitely generated if and only if it is finitely generated as a monoid.)  This also proves that various other subgroups of $\mathcal{R}$ are not finitely generated, such as:
\begin{enumerate}
\item The subgroup of $\R$ generated by Thompson's group~$V$ and all of the~$f_p$.\smallskip
\item The subgroup of all measure-preserving elements of~$\mathcal{R}$.\smallskip
\item The subgroup of $\R$ generated by Thompson's group $V$ and all synchronous rational functions.\smallskip
\item The subgroup of all bilipschitz elements of~$\R$.
\end{enumerate}
Note that (3) is a proper subgroup of (4).  For example, the element $f\in\R$ satisfying $f=(x_0,f)$ (where $x_0$ is the element shown in Figure~\ref{fig:Transducerx0}) is bilipschitz but does not lie in~(2).

We now turn to the proof of Theorem~\ref{thm:NotFG}.  Given a binary asynchronous transducer $T = (S,s_0,t,o)$, recall that a state $s\in S$ is \newword{accessible} if there exists an $\alpha\in \{0,1\}^\tast$ such that
\[
t(s_0,\alpha) = s.
\]
A \newword{cycle} in $T$ is an ordered pair $(c,\gamma)$, where $c\in S$ and $\gamma\in\{0,1\}^\tast$ is a nonempty binary sequence satisfying
\[
t(c,\gamma) = c.
\]
Such a cycle corresponds to a directed cycle of edges in the state diagram for~$T$. The length $|\gamma|$ of $\gamma$ is called the \newword{length} of the cycle.  A cycle $(c,\gamma)$ is \newword{accessible} if $c$ is an accessible state.

\begin{definition}Let $p$ be a prime.
\begin{enumerate}
\item We say that a transducer $T$ is \newword{oblivious} to $p$ if there exists an accessible cycle in $T$ whose length is not a multiple of~$p$.\smallskip
\item We say that a rational homeomorphism $f\in\R$ is \newword{oblivious} to $p$ if there exists a transducer for $f$  that is oblivious to~$p$.
\end{enumerate}
\end{definition}

Note that any transducer with fewer than $p$ states is automatically oblivious to~$p$.\pagebreak

\begin{lemma}\label{lem:fpIsNotOblivious}If $p$ is prime, then $f_p$ is not oblivious to $p$.
\end{lemma}
\begin{proof}
Let $T=(S,s_0,t,o)$ be any transducer for~$f$, and let $(c,\gamma)$ be an accessible cycle for~$T$. Since $c$ is accessible, there exists an $\alpha\in\{0,1\}^\tast$ so that $t(s_0,\alpha)=c$. Fix a $\lambda \in \{0,1\}^\tast$ so that $|\alpha| + |\gamma| + |\lambda|$ is a multiple of $p$, and let $\rho = 0^{p-1}1$.  Then the infinite binary sequence
\[
\alpha\gamma\lambda\rho^\omega
\]
eventually has $1$'s in positions that are multiples of~$p$, so $f_p(\alpha\gamma\lambda\rho^\omega)$ ends in an infinite sequence of~$0$'s.  In particular,
\[
f_p(\alpha\gamma\lambda\rho^\omega) = \beta\delta\mu 0^\omega
\]
where $\beta = o(s_0,\alpha)$, $\delta = o(c,\gamma)$, and $o(c,\lambda\rho^\omega) = \mu 0^\omega$ for some finite binary sequence $\mu\in\{0,1\}^\tast$.  If we now eliminate the trip around the cycle, we see that
\[
f_p(\alpha\lambda\rho^\omega) = \beta\mu 0^\omega.
\]
This ends with an infinite sequence~$0$'s, so $\alpha\lambda\rho^\omega$ must eventually have $1$'s in positions that are multiples of~$p$.  Then $|\alpha|+|\lambda|$ must be a multiple of~$p$, and therefore $|\gamma|$ is a multiple of~$p$.  Since the cycle $(c,\gamma)$ was arbitrary, we conclude that~$T$ is not oblivious to~$p$.  Since $T$ was arbitrary, it follows that $f_p$ is not oblivious to~$p$.
\end{proof}

\begin{lemma}\label{lem:ObliviousProduct}Let $p$ be a prime and let $f,f'\in\mathcal{R}$.  If $f$ is oblivious to $p$ and $f'$ has a transducer with fewer than $p$ states, then $f'f$ is oblivious to~$p$.
\end{lemma}
\begin{proof}Let $T=(S,s_0,t,o)$ and $T'=(S',s_0',t',o')$ be transducers for $f$ and~$f'$, where $T$ is oblivious to $p$ and $T'$ has fewer than $p$ states.  Then there exists a transducer $T''$ for $f'f$ with state set $S\times S'$ and initial state $(s_0,s_0')$ whose transition and output functions $t'',o''$ satisfy
\[
t''\bigl((s,s'),\alpha\bigr) = \bigl(t(s,\alpha),t'(s',\beta)\bigr)
\qquad\text{and}\qquad
o''\bigl((s,s'),\alpha\bigr) = o'(s',\beta)
\]
for all $\alpha\in\{0,1\}^\tast$, where $\beta = o(s,\alpha)$.

Now, since $T$ is oblivious to $p$, there exists an accessible cycle $(c,\gamma)$ for~$T$ whose length is not a multiple of~$p$. Since $c$ is accessible, there exists an $\alpha\in\{0,1\}^\tast$ such that
$t(s_0,\alpha) = c$. Then
\[
t(s_0,\alpha\gamma^n) = c
\]
for all $n\geq 0$. Let $\beta = o(s_0,\alpha)$ and $\delta = o(c,\gamma)$, so
\[
o(s_0,\alpha\gamma^n) = \beta\delta^n
\]
for all $n\geq 0$.  Since $T'$ has fewer than $p$ states, by the pigeonhole principle there exist numbers $j,k$ with $0\leq j < k < p$ such that
\[
t'(s'_0,\beta\delta^j) = t'(s'_0,\beta\delta^k).
\]
Let $c' = t'(s_0',\beta\delta^j)$, and observe that
\[
t'(c',\delta^{k-j}) = c'.
\]
Then $(c,c')$ is an accessible state for $T''$ since
\[
t''\bigl((s_0,s_0'),\alpha\gamma^j\bigr) = \bigl(t(s_0,\alpha\gamma^j),t'(s_0',\beta\delta^j)\bigr) =  (c,c').
\]
Moreover
\[
t''\bigl((c,c'), \gamma^{k-j}\bigr) = \bigl(t(c,\gamma^{k-j}),t'(c',\delta^{k-j})\bigr) = (c,c')
\]
so $\bigl((c,c'),\gamma^{k-j}\bigr)$ is an accessible cycle in $T''$.  But neither $|\gamma|$ nor $k-j$ is multiple of~$p$, so the length $|\gamma^{k-j}|  = (k-j)\,|\gamma|$ of this cycle is not a multiple of~$p$.  We conclude that $T''$ is oblivious to~$p$, so $f'f$ is oblivious to~$p$.
\end{proof}

\begin{proof}[Proof of Theorem~\ref{thm:NotFG}]Let $\mathcal{M}$ be a submonoid of $\mathcal{R}$ that contains infinitely many~$f_p$. For each $n\in\mathbb{N}$, let $\mathcal{M}_{\leq n}$ be the submonoid of $n$ generated by all elements of $\mathcal{M}$ that can be represented by transducers with $n$ or fewer elements.  This gives us an ascending sequence of monoids
\[
\mathcal{M}_{\leq 1} \subseteq \mathcal{M}_{\leq 2} \subseteq \mathcal{M}_{\leq 3} \subseteq \cdots
\]
with $\bigcup_{n\in\mathbb{N}} \mathcal{M}_{\leq n} = \mathcal{M}$.  If $p$ is prime, then by Lemma~\ref{lem:ObliviousProduct} every element of $\mathcal{M}_{\leq p-1}$ is oblivious to~$p$.  By Lemma~\ref{lem:fpIsNotOblivious}, it follows that $f_p\notin \mathcal{M}_{\leq p-1}$ for each~$p$, and therefore $\mathcal{M}$ is an ascending union of proper submonoids.
\end{proof}

\section*{Acknowledgements}

The authors would like to thank Collin Bleak for several helpful conversations.

\bigskip
\newcommand{\doi}[1]{\href{https://doi.org/#1}{Crossref}}
\newcommand{\arXiv}[1]{\href{https://arxiv.org/abs/#1}{arXiv}}
\newcommand{\biblink}[2]{\href{#1}{#2}}
\bibliographystyle{plain}

\end{document}